\documentclass[11pt, reqno]{amsart}

\usepackage{amssymb}
\usepackage[margin=2.5cm]{geometry}
\usepackage{hyperref}
\usepackage{csquotes}

\newenvironment{itquote}
{\begin{quote}\itshape}
	{\end{quote}}

\newtheorem{theorem}{Theorem}[section]
\newtheorem{lemma}[theorem]{Lemma}
\newtheorem{lem}[theorem]{Lemma}
\newtheorem{conj}[theorem]{Conjecture}
\newtheorem{prop}[theorem]{Proposition}

\newtheorem{cor}[theorem]{Corollary}
\newtheorem{remark}[theorem]{Remark}

\theoremstyle{definition}
\newtheorem{definition}[theorem]{Definition}

\newcommand{\h}{\mathcal{H}}
\newcommand{\hp}{\mathcal{H}_{n,p}}

\newcommand{\M}{\mathcal{M}}
\newcommand{\N}{\mathcal{N}}
\newcommand{\C}{\mathcal{C}}
\newcommand{\Mn}{(\mathcal{M}_n)_{n\geq 2}}
\newcommand{\Ktt}{K_{3,3}}

\newcommand{\R}{\mathcal{R}}
\newcommand{\ntwo}{\binom{[n]}{2}}

\numberwithin{equation}{section}

\begin{document}

\title{On plane rigidity matroids}



\author{Mykhaylo Tyomkyn}
\address{Department of Applied Mathematics, Faculty of Mathematics and Physics, Charles University}
\curraddr{}
\email{tyomkyn@kam.mff.cuni.cz}
\thanks{Supported by GA\v{C}R grant 25-17377S}


\begin{abstract}
We establish new properties of matroids and matroidal families associated with rigidity in dimension $2$, including the generic rigidity matroid family $\mathcal{R}$ and Kalai's hyperconnectivity matroid family $\mathcal{H}$. 

Answering a question of Kalai~\cite{Kal85} in a strong form, we show that all connected cubic graphs, with exceptions of $K_4$ and $K_{3,3}$, are independent in every $2$-rigidity family. We also prove that $\mathcal{R}$ is the unique matroidal $2$-rigidity family in which $K_{3,3}$ is not a circuit. As a geometric corollary of this result and the Bolker-Roth theorem~\cite{BR80}, it follows that $\mathcal{H}$ and $\mathcal{R}$ are the only $2$-rigidity families associated with algebraic curves in $\mathbb{R}^2$. 

Bernstein~\cite{Ber17} used tropical geometry to characterize $\mathcal{H}$-independent graphs as those admitting an edge-ordering without directed cycles and alternating closed trails. We provide a combinatorial proof of the sufficiency direction, extending Bernstein's theorem to positive characteristic. It follows that the wedge power matroid of $n$ generic points in dimension $n-2$ does not depend on the field characteristic. 

As a corollary, we obtain a new property of cubic graphs: every connected cubic graph except $K_4$ and $K_{3,3}$ has an orientation without directed and alternating cycles. The current proof of this purely graph theoretic statement relies on tropical geometry and matroid theory.
\end{abstract}
\maketitle

\section{Introduction}\label{sec:intro}

A graph embedded in $\mathbb{R}^d$ as a framework of bars and joints is said to be \emph{rigid} if it does not admit any continuous motion other than the isometries of the entire space. The study of rigidity in Euclidean space is a classical discipline in mathematics, physics and engineering, dating back to the work of Euler~\cite{euler1776nova}, Cauchy~\cite{Cau13} and Maxwell~\cite{Max64}. The 1970s saw the emergence of \emph{combinatorial rigidity theory}, which, via the notion of infinitesimal rigidity, allows rigidity questions to be studied using methods from matroid theory and graph theory. This led to the notion of abstract rigidity matroids.

Abstract $d$-rigidity matroids are combinatorial objects defined in a way to emulate geometric rigidity properties of $d$-dimensional bar-joint frameworks. Besides their central role in rigidity theory, 
they have been increasingly used in modern statistics and information theory~\cite{Ber26, bernstein2024maximum, blekherman2019maximum, Bra25, Bra24, gopalan2017maximally, gross2018maximum}. It is therefore important that we understand their combinatorial properties.

All graphs in this paper are assumed to be finite. A \emph{graph matroid} is a matroid defined on the ground set $\binom{[n]}{2}$ interpreting the elements as edges and sets of elements as graphs on $[n]$ without isolated vertices. 
The following characterization was proved by Nguyen~\cite{Ngu10} to be equivalent to Graver's~\cite{Gra91} original definition. 
\begin{definition}
Let $d\geq 1$. A matroid $\M_n$ on $\binom{[n]}{2}$ is an (abstract) \emph{$d$-rigidity matroid} if the following conditions are satisfied.
\begin{enumerate}
	\item The rank of $\M_n$ is $dn-\binom{d+1}{2}$.
	\item Every copy of $K_{d+2}$ is an $\M_n$-circuit.
\end{enumerate}  
\end{definition}
\noindent
While for $d=1$ for each $n$ there is only example -- the cycle matroid of the complete graph $K_n$, a much richer theory is encountered already for $d=2$. Two main examples of $2$-rigidity matroids are the generic plane rigidity matroid $\R_n$ and Kalai's $2$-hyperconnectivity matroid $\h_n$ (we omit a second index $d$ for the dimension, as this will be $2$ throughout). They will be the focus of this paper.  

The matroid $\R_n$ is initially defined in terms of rigidity of generic bar-joint frameworks in the plane. However, it has multiple equivalent combinatorial descriptions, and we shall use the following one, due to Pollaczek-Geiringer~\cite{PG27}  and Laman~\cite{Lam70}, as the definition. For a verification of the matroid axioms and a more detailed account of rigidity in $\mathbb{R}^2$ we refer to~\cite{gra93textbook}. 
\begin{definition}\label{def:R}
Let $\R_n$ be the matroid on $\binom{[n]}{2}$ as follows. A graph $G\subseteq \binom{[n]}{2}$ is independent if and only if every subset of $2\leq m\leq n$ vertices induces at most $2m-3$ edges. 
\end{definition}
\noindent

The hyperconnectivity matroid $\h_n$ was originally defined by Kalai~\cite{Kal85} in terms of the exterior algebra of $n$-dimensional vector spaces. It is also known to be the algebraic matroid of skew-symmetric $n\times n$ matrices of rank at most $2$, related to the Pl\"ucker embedding of the Grassmannian $Gr(2,\mathbb{R})$ in $\mathbb{R}^{\binom{n}{2}}$. Geometrically, $\h_n$ is the infinitesimal rigidity matroid of $n$ points on any non-degenerate conic in the plane, see~\cite{Cre23} for more detail. 
We shall use the following explicit definition, known to be equivalent. 
\begin{definition}\label{def:hyperconn}
Let $K=\mathbb{Q}(r_1,\dots,r_n,b_1,\dots,b_n)$ 
and for $i=1,\dots,n$ let
$\mathbf{p_i}=(r_i,b_i)\in K^2$ (equivalently, one could work over $K=\mathbb{R}$ or $\mathbb{C}$, taking $\bf{p}_1,\dots,p_n$ to be $n$ generic points in $K^2$). Let $\h_n$ be the row matroid of the $\binom{n}{2}\times 2n$ matrix $M_n$ over $K$ given by 
\[M_n=\begin{pmatrix}
	\bf{p}_2 & -\bf{p}_1 & 0 &\dots&0 & 0\\
	\bf{p}_3 & 0 & -\bf{p}_1 &\dots&0 & 0\\
	\vdots & \vdots & \vdots & \dots & \vdots &\vdots\\
	\bf{p}_n & 0 & 0         &\dots& 0 & -\bf{p}_1\\
	0 & \bf{p}_3 & -\bf{p}_2 &\dots& 0 & 0\\
	\vdots & \vdots & \vdots & \dots & \vdots &\vdots\\
	0 & 0 & 0 & \dots & \bf{p}_n & -\bf{p}_{n-1}	
\end{pmatrix}
\]
That is, for $1\leq i<j\leq n$, the row indexed with $\{i,j\}$ has in columns $2i-1,2i,2j-1,2j$ the entries $r_j,b_j,-r_i,-b_i$ respectively, and entry $0$ otherwise.
\end{definition}
\subsection{Matroidal $2$-rigidity families}
Both $\R=(\R_n)_{n\geq 2}$ an $\h=(\h_n)_{n\geq 2}$ are matroidal families in the following sense. Let $v(G)$ denote the number of vertices of $G$.
\begin{definition}
A graph matroid $\M_n$ is \emph{symmetric} if for every $\M_n$-independent graph $G\subseteq\binom{[n]}{2}$ every isomorphic copy of $G$ is also $\M_n$-independent. 
A sequence of symmetric graph matroids $\M=(\M_n)_{n\geq n_0}$ is a \emph{matroidal family} if every $G\subseteq \binom{\mathbb N}{2}$ that is $\M_n$-independent for some $n$ is $\M_n$-independent for all $n\geq v(G)$.
\end{definition}
\noindent
Note that circuits in matroidal families are similarly symmetric and hereditary. We can therefore,  
slightly abusing notation, write ``$G$ is $\M$-independent/dependent/circuit". 
\begin{definition}
A \emph{$d$-rigidity family} is a matroidal family $\M=\Mn$ in which each $\M_n$ is of rank $dn-\binom{d+1}{2}$. 
\end{definition}
\noindent
This automatically implies that $K_{d+2}$ is a circuit and therefore each $\M_n$ is a $d$-rigidity matroid. So, $d$-rigidity families are hereditary families of symmetric $d$-rigidity matroids. 

The $d$-rigidity families were first defined by Kalai under the name ``hypergraphic sequences" in the same seminal paper~\cite{Kal85} that introduced hyperconnectivity. To quote from it: 
\begin{itquote} While the only $1$-hypergraphic sequence is the sequence of graphic matroids there seems to be a rich class of $k$-hypergraphic sequences of matroids (including $\h_k^n$ and $\R_k^n$) which deserves a further study.
\end{itquote}	
The smallest graph distinguishing between the $2$-rigidity families $\R$ and $\h$ is $K_{3,3}$, which is independent in the former, but a circuit in the latter. We prove that $K_{3,3}$ is in fact a circuit in \emph{every} $2$-rigidity family apart from $\R$.
\begin{theorem}\label{thm:Laman1}
In every $2$-rigidity family $\M\neq \R$, $\Ktt$ is a circuit.
\end{theorem} 
 \noindent
It is known that any irreducible algebraic plane curve $C$ of degree at least $2$ gives rise to a $2$-rigidity family $\M(C)$ via the row matroids of the rigidity matrix for generically chosen points on $C$. As we have mentioned, when $C$ is a conic, we have $\M(C)=\h$. Theorem~\ref{thm:Laman1} implies that if $K_{3,3}$ is infinitesimally rigid on $C$, then $\M(C)=\R$. 
A classical theorem of Bolker and Roth~\cite{BR80} states that this is always the case when the points are not on a conic. Thus we obtain
\begin{cor}\label{cor:curves}
For every irreducible algebraic plane curve $C$ of degree $d \geq 2$ we have
$$\M(C)=\begin{cases}
	\h, & \text{if $d=2$}\\
	\R, & \text{if $d\geq 3$}.
\end{cases}
$$
\end{cor}
\noindent
Kalai (\cite{Kal85}, Problem 10.4) asked if every $3$-connected cubic planar graph is a circuit in some $2$-rigidity family. Using a similar idea as in the proof of Theorem~\ref{thm:Laman1}, we reach the diametrically opposite conclusion, even without the planarity assumption.
\begin{theorem}\label{thm:cubic}
Every 
connected cubic graph $H\neq K_4,K_{3,3}$ is independent in every $2$-rigidity family.
\end{theorem}
\noindent
Note that $K_4$ is by definition always a circuit and $K_{3,3}$ is by Theorem~\ref{thm:Laman1} independent in $\R$ and a circuit otherwise. Thus our results give a complete classification of cubic graphs with respect to $2$-rigidity families.

In particular, Theorem ~\ref{thm:cubic} implies that every such graph $H$ is $\h$-independent. Using Bernstein's theorem (Theorem~\ref{thm:Bernorig} below) we obtain the following graph theoretical corollary, which may have a direct combinatorial proof, but we could not see one quickly. 
\begin{cor}\label{cor:cubicorient}
Every connected cubic graph $H\neq K_4,K_{3,3}$ has an orientation without directed or alternating cycles.
\end{cor}
\subsection{Combinatorial properties of $\h_n$}
The matroids $\h_n$ are not as well understood combinatorially as $\R_n$. Despite this, Bernstein~\cite{Ber17}, using tools from tropical geometry, gave a surprising combinatorial characterization of $\h_n$-independent graphs. 

An \emph{alternating closed trail} in an oriented graph is a cyclic sequence of distinct arcs $e_0,\dots,e_{2m}=e_0$ such that, modulo $2m$, for each $i$ the arcs $e_{i}$ and $e_{i+1}$ share a vertex $v_i$ and have the same direction with respect to it, and for each $i$ the vertices $v_i$ and $v_{i+1}$ are distinct. Note that non-consecutive vertices $v_i$ and $v_j$ may coincide.

Let us call an oriented graph~\emph{Bernstein} if it has no directed cycles and no alternating closed trails. Call an undirected graph \emph{Bernstein-orientable} if it admits such an orientation.
\begin{theorem}[\cite{Ber17}]\label{thm:Bernorig}
	$G\subseteq \binom{[n]}{2}$ is $\h_n$-independent if and only if it is Bernstein-orientable.
\end{theorem}
\noindent
Remarkably, no combinatorial proof of Theorem~\ref{thm:Bernorig} is known to date. However, for bipartite graphs Brakensiek, Dhar, Gao, Gopi and Larson~\cite{Bra24} gave recently a combinatorial proof of the sufficiency direction and extended Bernstein's theorem to positive characteristic. 

Let $p$ be prime. Define $\h_n(p)$ to be the counterpart of $\h_n$ in characteristic $p$. That is, consider the matrix $M_n$ as in Definition~\ref{def:hyperconn} but using $\mathbb{F}_p$ in place of $\mathbb{Q}$, and let $\hp$ be its row matroid (by convention, let us write $\h_n(0)=\h_n$).
\begin{theorem}[\cite{Bra24}]
	Every Bernstein-orientable bipartite graph $G\subseteq \binom{[n]}{2}$ is $\hp$-independent for all $p$.
\end{theorem}
\noindent
We give a combinatorial proof and a similar strengthening of the full sufficiency direction.
\begin{theorem}\label{thm:bernsteincharp}
Every Bernstein-orientable graph $G\subseteq \binom{[n]}{2}$ is $\hp$-independent for all $p$.  	
\end{theorem}	
\noindent
Observe that the determinants of the minors of $M_n$ are polynomials with integer coefficients. Hence, when considered in characteristic $p>0$, they are modulo $p$ reductions of their counterparts in characteristic $0$. Consequently, every $\hp$-independent graph is $\h_n$-independent. Combining this with Theorem~\ref{thm:Bernorig} and our Theorem~\ref{thm:bernsteincharp} proves that these matroids are identical. 
\begin{cor}\label{cor:samematroidallp}
	For all $p$ we have $\hp=\h_n$. 
\end{cor}
\noindent
The wedge power matroid $W_n(r,p)$ is the linear matroid of the vectors $\{v_i\wedge v_j:1\leq i<j\leq n\}\subseteq \bigwedge V^2$ where $v_1,\dots,v_n$ are $n$ generic vectors in an $r$-dimensional vector space $V$ over an infinite field of characteristic $p$. Brakensiek et.~al.~\cite{Bra24} established a duality relation between the wedge power and hyperconnectivity matroids --- Lov\'asz's proof~\cite{Lovasz77} of the skewed version of Bollob\'as's two families theorem can be viewed a manifestation of this duality. In particular, it was shown in~\cite{Bra24} that the dual of $\hp$ is $W_n(n-2,p)$.  
Our result therefore implies that the latter does not depend on the field characteristic. 
\begin{cor}
	The wedge power matroid $W_n(n-2,p)$ does not depend on $p$. 
\end{cor}	
 \noindent
Since $\R$ and $\h$ are the only known examples of $2$-rigidity families, it is natural to ask whether further $2$-rigidity families exist. Note that our results close two potential avenues of constructing new $2$-rigidity families -- using algebraic curves, and changing the field characteristic.
They would not to give anything new by Corollaries~\ref{cor:curves} and~\ref{cor:samematroidallp}, respectively. Due to these considerations, as well as small graph experiments, we believe that no further families exist. 
 \begin{conj}\label{conj:only2}
 	There exist only two $2$-rigidity families: $\R$ and $\h$.
 \end{conj}

The rest of the paper is organized as follows. In Section~\ref{sec:H} we give a combinatorial proof of the sufficiency direction of Theorem~\ref{thm:Bernorig} and prove Theorem~\ref{thm:bernsteincharp}. In Section~\ref{sec:R} we collect some properties of matroidal families. In Sections~\ref{sec:U} and~\ref{sec:C} we prove Theorems~\ref{thm:Laman1} and \ref{thm:cubic}, respectively. We conclude in Section~\ref{sec:D} with a discussion.  
\subsection*{Notation.}
We associate graphs without isolated vertices with their edge sets. Consequently, for a graph $G$ let $|G|$ denote its number of edges. We denote by $V(G)$ the vertex set of $G$ and let $v(G)=|V(G)|$. We say ``$G$ contains vertex $v$'' if $v\in V(G)$ (equivalently, if $v$ is contained in an edge of $G$).
For a subset $U\subseteq V(G)$ we denote by $G[U]$ the induced subgraph of $U$. Let $\deg_G(v)$ be the degree of vertex $v$ in $G$. By $G-v$ we denote the graph $G$ with the vertex $v$ deleted. For an edge $e$ of $G$ we write $G-e$ to denote the graph $G$ with $e$ deleted. For a missing edge $e\in \binom{V(G)}{2}\setminus G$, $G+e$ denotes $G$ with $e$ added.  

When discussing oriented graphs, for a better distinction we speak of \emph{arcs} rather than edges. The out-degree and in-degree of a vertex $v$ in $D$ are denoted $\deg^+_D(v)$ and $\deg^-_D(v)$, respectively. Let $\Delta^+(D)$ and $\Delta^-(D)$ denote the largest out- and in-degree in $D$, respectively.

\subsection*{Statement on AI use} 
Work on this paper predates the wide-scale use of AI in mathematical research. Free versions of ChatGPT and Google Gemini were used for minor copy edits in the introduction. Otherwise this paper is AI-free.

\section{Hyperconnectivity}\label{sec:H} 
\noindent
In this section we prove Theorem~\ref{thm:bernsteincharp}. 

The \emph{degree sequence} $g(D)$ of an oriented graph $D$ is the tuple $g(D)=(\deg_D^+,\deg_D^-)$, where $\deg_D^+: V(D)\rightarrow \mathbb{Z}$ encodes the out-degrees of the vertices, and $\deg_D^-$ the in-degrees.
The following fact is well-known, but we include its short proof for completeness.
\begin{prop}\label{prop:colourblind}
Let $D$ be an acyclic orientation of an undirected graph $G$ and $D'$ an orientation of $G$ with $g(D')=g(D)$. Then $D'=D$.
\end{prop}
\begin{proof}
Suppose for a contradiction that $D'$ is distinct from $D$. Consider $\tilde{D}=D\setminus D'$, i.e. the oriented graph of all arcs in $D$ that have the opposite orientation in $D'$. This is a non-empty subgraph of $D$ in which each vertex $v$ satisfies $$\deg_{\tilde{D}}^+(v)= \deg_{\tilde{D}}^-(v).$$ 
Take the longest directed path $\tilde{D}$: its last vertex must feed an arc back to the path, resulting in a directed cycle in $\tilde{D}$ and $D$ -- a contradiction. 
\end{proof}
\noindent
Let us extend this to bi-coloured orientations. 
\begin{lemma}\label{lem:sameorientdiffcols}
Suppose $D$ is an oriented graph that admits two distinct red/blue colourings (partitions) $(R,B)$ and $(R',B')$ such that for each vertex $v$,
$$\deg^+_R(v)=\deg^+_{R'}(v), \ \ \deg^-_R(v)=\deg^-_{R'}(v), \ \
\deg^+_B(v)=\deg^+_{B'}(v),  \text{ and} \ \deg^-_B(v)=\deg^-_{B'}(v).
$$
Then $D$ contains an alternating closed trail.
\end{lemma}
\begin{proof}
Since $(R,B)$ and $(R',B')$ are partitions, we have
$$R'\setminus R=R'\cap B=B\setminus B'.
$$	
Let 
$$
\tilde{D}=(R\setminus R')\cup (B\setminus B'),
$$ 
i.e. the edges of different colour in $\C$ and $\C'$. By assumption, $\tilde{D}$ is not empty.
For every vertex $v\in V(\tilde{D})$ we have 
$$\deg^+_{R\setminus R'}(v)=\deg^+_{R}(v)-\deg^+_{R\cap R'}(v)=\deg^+_{R'}(v)-\deg^+_{R\cap R'}(v)=\deg^+_{R'\setminus R}(v)=\deg^+_{B\setminus B'}(v),  
$$
and so 
$$\deg^+_{\tilde{D}}(v)=\deg^+_{R\setminus R'}(v)+\deg^+_{B\setminus B'}(v)=2\deg^+_{R\setminus R'}(v),
$$
meaning each out-degree in $\tilde{D}$ is even, and the same holds for in-degrees analogously.

Now, consider an arbitrary vertex $v$ of non-zero out-degree in $\tilde{D}$. Take an outgoing arc $(v,u)$, and starting from it perform an alternating walk $v\rightarrow u \leftarrow \dots$, using distinct arcs, until this is no longer possible. By parity of the degrees, this will be the case when the walk returns to $v$, resulting in a closed alternating trail $\dots \leftarrow v \rightarrow u \leftarrow \dots$ in $\tilde{D}$ and $D$.
\end{proof}
\begin{definition}	
	Let $G$ be a graph without isolated vertices.
	A \emph{configuration} of $G$ is a partition (red/blue colouring) $\C=(R,B)$ of an  orientation of $G$.  
	The \emph{degree function} $f(\C)$ is the tuple 
	$$f(C)=(\deg^+_R,\deg^-_R,\deg^+_B,\deg^-_B),$$ 
	of functions encoding the in- and out-degrees in $\C$ of both colours. A configuration $\C$ of $G$ is \emph{recoverable} if there is no other configuration $\C'$ of $G$ with $f(\C')=f(\C)$.
\end{definition}	
\begin{lemma}\label{lem:Bernrec}
Suppose $\mathcal{C}=(R,B)$ is a configuration of a graph $G$ such that $D=R\cup B$ is Bernstein. Then $\C$ is recoverable.
\end{lemma}
\begin{proof}
Suppose for a contradiction that another configuration $\C'=(R',B')$ satisfies $f(\C')=f(\C)$. 
Unifying the colours, we obtain that $D'=R'\cup B'$ satisfies
$g(D')=g(D)$. Since $D$ is acyclic, by Proposition~\ref{prop:colourblind} we must have $D'=D$. So, $\C$ and $\C'$ satisfy the assumptions of Lemma~\ref{lem:sameorientdiffcols}, and from it we deduce that $D$ has an alternating closed trail, a contradiction.
\end{proof}
\begin{definition}
A graph $G$ is $UFP$ (uniquely forest partitionable) if there exists a recoverable configuration $\C=(R,B)$ of $G$ such that
$\max(\Delta^+(R),\Delta^+(B))\leq 1$.
\end{definition}
\noindent
Note that in $\C$ as above, $R$ and $B$ are directed pseudoforests. In order for a configuration to be recoverable, both must actually be forests, since reversing all orientations on a monochromatic directed cycle preserves the degree function. The following observation was made in \cite{Bra24}. 
\begin{prop}
Every $UFP$ graph is $\hp$-independent for all $p$.
\end{prop}
\noindent
To briefly summarize its proof -- it is claimed that the matrix $M_n[G]$, the restriction of $M_n$ to the rows of $G$ has rank $|G|$ in every characteristic. It follows from Definition~\ref{def:hyperconn} that permutation diagonals in $|G|\times |G|$-minors of $M_n[G]$ naturally correspond to configurations $(R,B)$ of $G$ with $\Delta^+(R),\Delta^+(B)\leq 1$, whereby the $r_i$ and $b_i$ variables encode the red and blue edges, respectively. So, if a configuration is recoverable, its monomial will appear in the determinant expansion exactly once, and therefore the determinant will be non-zero.  

Hence, in order to prove Theorem~\ref{thm:bernsteincharp} it suffices to show that every Bernstein-orientable graph is $UFP$. 
So, let $G$ be a Bernstein-orientable graph on $[n]$, and take its Bernstein orientation $D$. Using the standard correspondence between acyclic orientations and vertex orderings, we may assume without restriction that in $D$ every edge $\{i,j\}$ for $i<j$ is oriented from $i$ to $j$. 

Define the auxiliary undirected graph $F$ as follows. 
$$V(F)\colon= \{a_1^+, a_n^-\}\cup\bigcup_{i=2}^{n-1}\{a_i^+,a_i^-\},$$
$$E(F)\colon=\{\{a_i^+,a_j^-\}: (i,j)\in E(D)\}.$$ 
\begin{lemma}
$F$ is a forest. 
\end{lemma}
\begin{proof}
Note that $F$ is bipartite under the vertex partition $\{a_1^+,\dots,a^+_{n-1}\}\cup\{a_2^-,\dots,a_n^-\}$. Suppose $F$ contains a cycle. It will be of the form $$a_{i_0}^+,a_{i_1}^-,\dots,a_{i_{2m-1}}^-,a_{i_{2m}}^+=a_{i_0}^+.$$ 
Then $D$ would contain the cyclic sequence of arcs
$$i_0\rightarrow i_1 \leftarrow \dots \rightarrow i_{2m-1} \leftarrow i_{2m}=i_0.
$$
in which the even indexed vertices are distinct, and the odd indexed ones are also distinct. Since each arc goes from an even to an odd indexed vertex, all the arcs must be distinct. Therefore this sequence forms an alternating closed trail in $D$, a contradiction. 
\end{proof}
\noindent
Hence, as a forest, $F$ admits an orientation of maximum out-degree at most $1$. Let the auxiliary oriented graph $\vec{F}$ be such an orientation of $F$.

Now, for every edge $(i,j)$ of $D$ (where $i<j$) we colour it \emph{blue} if $(a_i^+,a_j^-)\in E(\vec{F})$ and \emph{red} if $(a_j^-,a_i^+)\in E(\vec{F})$. 
This defines an edge-partition of $D$ into two oriented graphs $B$ and $R$, given by the blue and red edges, respectively. 

\begin{lemma}\label{lem:DeltaminusDeltaplus}
We have $\Delta^+(B)\leq 1$ and $\Delta^-(R)\leq 1$.
\end{lemma}
\begin{proof}
Suppose some $k\in [n]$ has two out-neighbours $i$ and $j$ in $B$ for some $i,j>k$. Then $(a_k^+, a_i^-)$ and $(a_k^+, a_j^-)$ are arcs of $\vec{F}$,
and so $a_k^+$ has two out-neighbours in $\vec{F}$, a contradiction. 	
	
Similarly, suppose $k$ has two in-neighbours $i$ and $j$ in $R$ for some $i,j<k$. Then $(a_k^-, a_i^+)$ and $(a_k^-, a_j^+)$ are arcs of $\vec{F}$, and so $a_k^-$ has two out-neighbours in $\vec{F}$, a contradiction. 
\end{proof}
\noindent
Note that by Lemma~\ref{lem:Bernrec} the configuration $\C=(R,B)$ of $G$ is recoverable since $R\cup B=D$ was Bernstein. For an oriented graph $J$ let $r(J)$ denote its reverse orientation. Note that $r$ is an involution: $r(r(J))=J$. Moreover,
$$g(r(J))=(\deg^+_{r(J)}, \deg^-_{r(J)})=(\deg^-_J, \deg^+_J). 
$$
In particular, by Lemma~\ref{lem:DeltaminusDeltaplus} we have $\Delta^+(r(R))\leq 1$. Therefore, $\bar\C=(r(R), B)$ is a configuration of $G$ with 
$$\max(\Delta^+(r(R)),\Delta^+(B))\leq 1.$$ 
\noindent
We now claim that $\bar\C$ is recoverable. Indeed, suppose for some configuration $\bar\C'=(R',B')$ of $G$ we have $f(\bar\C')=f(\bar\C)$. That is, 
$$(\deg^+_{R'},\deg^-_{R'},\deg^+_{B'},\deg^-_{B'})=(\deg^+_{r(R)},\deg^-_{r(R)},\deg^+_B,\deg^-_B)=(\deg^-_R,\deg^+_R,\deg^+_B,\deg^-_B).
$$
Then the configuration $\C'=(r(R'),B')$ of $G$ satisfies 
$$f(\C')=(\deg^-_{R'},\deg^+_{R'},\deg^+_{B'},\deg^-_{B'})=(\deg^+_R,\deg^-_R,\deg^+_B,\deg^-_B)=f(\C).$$
However, $\C$ is recoverable, so $\C'=\C$, implying $B'=B$ and $r(R')=R$. Hence, 
$$R'=r(r(R'))= r(R),$$ 
yielding $\bar\C'=\bar\C$. 
Therefore $G$ is $UFP$, and this completes the proof of Theorem~\ref{thm:bernsteincharp}.

The configuration $\bar{\C}$ from the proof might be of independent interest. Let us summarize its properties below.
\begin{cor}\label{cor:redblue}
A Bernstein-orientable (equivalently, $\h$-independent) graph $G$ can be red/blue-coloured such that
\begin{itemize}
	\item The red and blue graphs $R$ and $B$ are forests,
	\item There exist tree orders $<_{_R}$ and $<_{_B}$ on $V(G)$ such that $v<_{_R} w$ iff $w<_{_B} v$, and 
	\item There is no colour-alternating trail $$v_0>_{_R} v_1>_{_B} v_2>_{_R}\dots>_{_B} v_{2m}=v_0.$$ 
\end{itemize}
\end{cor}
\noindent
In fact, it is not hard to see that the converse also holds: a graph admitting a colouring as above is Bernstein-orientable. We omit the details.  

\section{Plane rigidity}\label{sec:R}
\noindent
In this section we collect some basic facts about $2$-rigidity matroids and families. 

It follows directly from Definition~\ref{def:R} that $\R_n$-circuits are graphs $G$ with $|G|=2n-2$ and such that every subset of $2\leq m<n$ vertices induces at most $2m-3$ edges. This in turn implies that $\R_n$ is maximal 
in the following sense.
\begin{prop}\label{prop:lamanmax}
Let $\M_n$ be a $2$-rigidity matroid. Then every $\M_n$-independent graph is $\R_n$-independent. Consequently, for any $2$-rigidity family $\M$, every $\M$-independent graph is $\R$-independent.
\end{prop}
\noindent
A \emph{Henneberg $0$-extension} of a graph $G$ is an addition of a new vertex of degree $2$. It is well-known (see e.g.~\cite{Ngu10}) that every $2$-rigidity matroid $\M_n$ has the \emph{$0$-extension property}. That is, if $G^+\subseteq \ntwo$ is a $0$-extension of an $\M_n$-independent graph $G$ then $G^+$ is also $\M_n$-independent. Therefore, the same holds for $2$-rigidity families $\M$: if $G$ is $\M$-independent, then so is $G^+$. We will also use it in the contrapositive: if $G^+$ is $\M$-dependent, then so is $G$.

By monotonicity, adding a new vertex of degree $1$ also preserves $\M$-independence. A graph is called \emph{$2$-degenerate} if it can be constructed from an edgeless graph by sequentially adding new vertices of degree $1$ or $2$. By the above, $2$-degenerate graphs are $\M$-independent. Call a graph \emph{properly subcubic} if its degrees are at most $3$ and it has a vertex of degree at most $2$. 
\begin{lem}\label{lem:subcubic}
	Let $\M$ be a $2$-rigidity family and $G$ a graph.
	\begin{enumerate}
		\item If $G$ is an $\M$-circuit then every vertex of $G$ has degree at least $3$.
		\item If $G$ is connected and properly subcubic, then $G$ is $\M$-independent.
		\item If $G$ is connected, cubic and $\M$-dependent, then $G$ is an $\M$-circuit.
	\end{enumerate}
\end{lem}	
\begin{proof}
	Claim (1) is a direct consequence of the $0$-extension property. 
	
	Concerning (2), observe that deleting a vertex of degree at most $2$ results in a graph with properly subcubic components. Hence, by induction, any such graph is $2$-degenerate.
	
	As for (3), note that for every edge $e\in G$ the graph $G-e$ has properly subcubic components, and therefore is $\M$-independent by (2). Hence, $G$ is an $\M$-circuit.
\end{proof}
\noindent
Two vertices in a graph are called \emph{twins} if they have the same neighbours (in particular, twins are not adjacent). Note that being twins is an equivalence relation giving rise to \emph{twin classes} of vertices. 
For a graph $G$ and a vertex  $v\in V(G)$ we denote by $G\oplus v$ the graph obtained by `cloning' $v$, i.e. adding a new twin of $v$.

The next lemma is a crucial ingredient in the proofs of Theorems~\ref{thm:Laman1} and~\ref{thm:cubic}.
\begin{lemma}[Tripling trick]
	\label{lem:tripling}
	Let $\M=\Mn$ be a $2$-rigidity family. Let $D$ be an $\M$-circuit and $v\in V(D)$ a degree $3$ vertex. Then
	\begin{enumerate}
		\item For any edge $e$ not incident with $v$ the graph $(D\oplus v)-e$ is $\M$-dependent. 
		\item For any degree $3$ vertex $w$ not adjacent to $v$ the graph $(D\oplus v)-w$ is $\M$-dependent.   
	\end{enumerate}
\end{lemma}
\begin{proof}
	We may assume that $V(D)=[n]$. In particular $D$ is an $\M_n$-circuit, and therefore also an $\M_{n+1}$-circuit.
	Let $D^+=D\oplus v$, where we can assume that the new twin of $v$ is the vertex $n+1$. Let $D'=D^+-v$. Since $D$ and $D'$ are isomorphic (we replaced $v$ with its twin $n+1$), by symmetry $D'$ is an $\M_{n+1}$-circuit. Let $e$ be an edge of $D$ not incident with $v$. Note that both $D$ and $D'$ contain $e$.
Hence, by the circuit elimination axiom in $\M_{n+1}$ 
we obtain that $D\cup D'-e=(D\oplus v)-e$ is $\M$-dependent.

	Let now $w$ be a degree $3$ vertex of $D$ not adjacent to $v$.
	If $v$ and $w$ are twins then $(D\oplus v)-w$ is isomorphic to $D$, and therefore $\M$-dependent. Assuming otherwise, since $v$ and $w$ have the same degree, there will be an edge $e$ incident with $w$ but not with $v$. By part (1), $(D\oplus v)-e$ is an $\M$-dependent graph in which $w$ has degree $2$. By the $0$-extension property, $((D\oplus v)-e)-w=(D-w)\oplus v$ must also be $\M$-dependent.  	
\end{proof}
\begin{remark}
	The operation $(G-w)\oplus v$ is known in extremal graph theory as the ``Zykov symmetrization''. It was  introduced in~\cite{Zyk49} in order to prove what now is known as Tur\'an's theorem, which Zykov discovered independently.
\end{remark}
\noindent
A \emph{Henneberg $1$-extension} is an addition of a new vertex of degree $3$ connected to three vertices $u,v,w$ of $G$ such that $uv$ is an edge, and deletion of the edge $uv$. The $1$-extensions are known to preserve independence in $\R$.
A \emph{suppression} is the inverse operation. That is, a vertex of degree $3$ whose neighbours do not form a triangle is deleted and a missing edge between two of its neighbours is added. While it is not true in general that suppressions preserve $\R$-independence, it still holds partially in the following sense. We say that a graph $G$ is an \emph{$\R$-base} if it is isomorphic to an $\R_n$-base for $n=v(G)$. In other words, $|G|=2v(G)-3$ and every subset of $2\leq m<v(G)$ vertices induces at most $2m-3$ edges (in the literature these graphs are also referred to as ``minimally generically rigid'' and ``isostatic'').
\begin{prop}[\cite{Tay85}, see also~\cite{Ber03}]\label{prop:suppr}
	Let $B$ be an $\R$-base and $v$ a degree $3$ vertex in $B$. Then there exists a suppression at $v$ resulting in an $\R$-base. 
\end{prop}
\noindent
We shall need the following consequence of the above. 
\begin{lemma}\label{lem:precircuit}
	Let $B$ be an $\R$-base and $v\in V(B)$ a vertex of degree $3$. Then there exists an edge $e\in \binom{V(B)}{2}\setminus B$ with $v\notin e$ such that the fundamental $\R$-circuit of $B+e$ contains $v$.   
\end{lemma}
\begin{proof}
	By Proposition~\ref{prop:suppr} there exists an edge $e\in \binom{V(B)}{2}\setminus B$ not incident with $v$ such that $(B+e)-v$ is an $\R$-base. Let $C$ be the fundamental circuit of $B+e$. Since $(B+e)-v$ is $\R$-independent, we must have $v\in C$.
\end{proof}
\section{Uniqueness of $\R$}\label{sec:U}
\noindent
In this section we prove Theorem~\ref{thm:Laman1}.
 
For a $2$-rigidity family $\M$, let us call any $\M$-circuit which is not an $\R$-circuit a \emph{short $\M$-circuit}.
\begin{lemma}\label{lem:shortcircuitlamanbase}
	Every short $\M$-circuit $C$ is $\R$-independent. In particular, $|C|\leq 2v(C)-3$ and $C$ contains at least six degree $3$ vertices. 
\end{lemma}
\begin{proof}
	Suppose that $C$ is an $\R$-dependent short $\M$-circuit. Since $C$ is not an $\R$-circuit, there is a proper subgraph $D\subsetneq C$ which is $\M$-independent and an $\R$-circuit, contradicting Proposition~\ref{prop:lamanmax}. So, $C$ is $\R$-independent, and by Definition~\ref{def:R}, we have $|C|\leq 2v(C)-3$. Consequently, the sum of all degrees in $C$ is at most $4v(C)-6$. Since, by Lemma~\ref{lem:subcubic}, each degree is at least $3$, by pigeonhole at least six vertices must be of degree exactly $3$. 
\end{proof}
\begin{theorem}\label{thm:short}
Suppose that $\M$ is a $2$-rigidity family in which $K_{3,3}$ is not a circuit. Then $\M$ has no short circuits.
\end{theorem}
\noindent
This readily implies Theorem~\ref{thm:Laman1} as follows. 
\begin{proof}[Proof of Theorem~\ref{thm:Laman1}]
Let $\M$ be a $2$-rigidity family where $K_{3,3}$ is not a circuit; we claim that $\M=\R$.	
By Proposition~\ref{prop:lamanmax} every $\M$-independent graph is $\R$-independent. To see that the converse also holds, suppose for a contradiction that there exists an $\R$-independent graph $A$, which is $\M$-dependent. Then a subgraph $C\subseteq A$ is an $\R$-independent $\M$-circuit. By Theorem~\ref{thm:short}, $C$ is 
	not short, so $C$ an $\R$-circuit, a contradiction. Therefore every $\R$-independent graph is $\M$-independent, and we conclude that $\M=\R$. 
\end{proof}	
\noindent
To prove Theorem~\ref{thm:short}, let $\M$ be as stated. Note that $K_{3,3}$ must be $\M$-independent, since $K_{3,3}$ with an edge removed is $2$-degenerate. Suppose for a contradiction that there exists a short $\M$-circuit.

We shall need the following graph parameter. For a graph $G$ and a vertex $v\in V(G)$ define $tw_G(v)$ to be the number of twins of $v$ (excluding $v$ itself). Define 
$$tw_3(G):=\max\{ tw_G(v): v\in V(G), \deg_G(v)=3\}.
$$ 
\noindent
Call a short $\M$-circuit $C$ \emph{critical} if $v(C)$ is minimal among all short $\M$-circuits, subject to this $|C|$ is minimal, and subject to the above $tw_3(G)$ is maximal. Clearly, if a short circuit exists, there exists a critical one. So, let $C$ be an arbitrary critical short $\M$-circuit. 
\begin{lemma}\label{lem:nolamaninside}
It is impossible that there is an $\R$-base $B\subseteq C$ such that $V(B)\neq V(C)$ and a vertex $v\in V(B)$ with $\deg_B(v)=\deg_C(v)=3$. 
\end{lemma}
\begin{proof}
	Suppose for a contradiction that we have $B$ and $v$ as stated. Note that $B$ is $\M$-independent since it is a proper subset of $C$. Moreover, by Lemma~\ref{lem:precircuit}, there is a missing edge $e\in \binom{V(B)}{2}\setminus B$ with $v\notin e$ such that there is an $\R$-circuit $D\subseteq B+e$ containing $v$ and $e$. Note that $e\notin C$, as otherwise we would have $D\subseteq C$, contradicting Lemma~\ref{lem:shortcircuitlamanbase}. 
	
	Observe that $D$ is $\M$-dependent by Proposition~\ref{prop:lamanmax}. Furthermore, $D$ cannot contain an $\M$-circuit $T$ as a proper subset. Indeed, $T$ would either be an $\R$-circuit or a short $\M$-circuit. The former cannot happen since $D$ itself is an $\R$-circuit. While the latter would imply $T$ is a short $\M$-circuit with 
	$$v(T)\leq v(D)\leq v(B)<v(C),
	$$  
	contradicting the minimality of $v(C)$ assumption. 
	Hence, $D$ is an $\M$-circuit. 
	
	Let $f$ be one of the three edges incident with $v$ in $B$ and $C$. Suppose $A = C+e-f$ is $\M$-independent. Then 
	$$A+f=C+e\supseteq D,$$
	meaning that $D$ is the fundamental $\M$-circuit of $A+f$. Since $e\in D$, it follows that $A+ f-e=C$ is $\M$-independent, a contradiction. We conclude that $A=C+e- f$ is $\M$-dependent. Note however that in $A$ the vertex $v$ has degree $2$. Hence, by the $0$-extension property, $A-v=C+e-v$ is $\M$-dependent.
	
	Thus, there exists an $\M$-circuit $S\subseteq C+e-v$. Note that $e\in S$, as otherwise $S\subsetneq C$, which is impossible for two $\M$-circuits. Moreover, since $v(S)<v(C)$, by criticality of $C$, $S$ cannot be a short $\M$-circuit, and so $S$ is an $\R$-circuit. 
	
	To summarize, $D$ and $S$ are $\M$ circuits, $\R$-circuits, $e\in D\cap S$, and they are distinct since $v$ is a vertex in $D$ but not in $S$. Hence, we may apply the circuit elimination axiom in $\M$ and in $\R$, to conclude that $(D\cup S) - e \subseteq C$ is $\M$-dependent and $\R$-dependent. Since $C$ is an $\M$-circuit, it follows that 
	$$(D\cup S) - e = C.$$ 
	So, $C$ is $\R$-dependent. However, $C$ is a short $\M$-circuit, which by Lemma~\ref{lem:shortcircuitlamanbase} is $\R$-independent, a contradiction.
\end{proof}

\begin{lemma}\label{lem:notwotwins}
$tw_3(C)\leq 1$.
\end{lemma}
\begin{proof}
	If some degree $3$ vertex $v$ has at least two twins, then $C$ contains a copy $B$ of $K_{3,3}$ (an $\R$-base) as a subgraph, with $v\in V(B)$ and $\deg_B(v)=\deg_C(v)=3$. By Lemma~\ref{lem:nolamaninside} this can only happen if $V(B)=V(C)$, meaning $v(C)=v(B)=6$. By Lemma~\ref{lem:shortcircuitlamanbase}, we obtain $$|C|\leq 2v(C)-3=9=|B|,$$ 
	and so $C=B$, contradicting our assumption that $K_{3,3}$ not an $\M$-circuit. 	
\end{proof}
\noindent
Let $v$ be a degree $3$ vertex of $C$ with $tw_C(v)=tw_3(C)$. By Lemma~\ref{lem:shortcircuitlamanbase} and Lemma~\ref{lem:notwotwins}, there exists a degree $3$ vertex $w\in V(C)$ which is neither $v$, nor its neighbour or twin. Fix an arbitrary such $w$. Define $$C'=(C\oplus v)-w,$$ 
and note that $v(C')=v(C)$ and $|C'|=|C|$. Let $v'$ be the new twin of $v$ in $C'$. Since the twin class of $v$ has been enlarged by $v'$, and by the choice of $v$, we have
$$tw_3(C')\geq tw_{C'}(v)\geq tw_C(v)+1=tw_3(C)+1>tw_3(C).$$
\begin{lemma}\label{lem:cprimeshort}
	$C'$ is a short $\M$-circuit. 
\end{lemma}
\begin{proof}
By Lemma~\ref{lem:tripling}, $C'$ is $\M$-dependent.
Consider an $\M$-circuit $D'\subseteq C'$. Since $C'-v'\subsetneq C$ is $\M$-independent we must have $v'\in V(D')$ and, by Lemma~\ref{lem:subcubic}, all three edges of $C$ incident with $v'$ must be in $D'$. Since $C'-v$ is isomorphic to $C'-v'$, we similarly have $v\in V(D')$ and all three edges of $C$ incident with $v$ are also in $D'$. 

Let $B=D' - v'$. Since $B\subsetneq D'$, it is $\M$-independent, and so, by Proposition~\ref{prop:lamanmax}, it is also $\R$-independent. Moreover, we have $v\in V(B)$.

If $D'$ is an $\R$-circuit then 
$$|B|=|D'|-3=2v(D')-2-3= 2(v(D')-1)-3=2v(B)-3,
$$
and $B$ is $\R$-independent, so $B\subseteq C$ is an $\R$-base. Note however that 
$$B=D'-v'\subseteq C'-v'=C-w,$$
so $V(B)\neq V(C)$. Moreover, $v\in V(B)$ and $\deg_B(v)=\deg_C(v)=3$. This is impossible by Lemma~\ref{lem:nolamaninside}. Hence, $D'$ is not an $\R$-circuit. 
	
So, $D'$ is a short $\M$-circuit. Since $C$ is critical, by minimality of $v(C)$ we must have $$v(D')\geq v(C)=v(C'),$$ which means $V(D')=V(C')$, as $D'\subseteq C'$. By minimality of $|C|$ subject to that, we must have $|D'|\geq |C|=|C'|$, so $D'=C'$. Thus $C'$ is a short $\M$-circuit, as claimed.
\end{proof}
\noindent
In summary, $C'$ is a short $\M$-circuit with $v(C')=v(C)$, $|C'|=|C|$ and $tw_3(C')>tw_3(C)$, contradicting criticality of $C$. This completes the proof of Theorem~\ref{thm:short}. 

\section{Cubic graphs}\label{sec:C}
\noindent
In this section we prove Theorem~\ref{thm:cubic}. We shall need the following fact. 
\begin{prop}[\cite{Jac24}, Lemma 6.11]\label{prop:circuits2conn}
	If $G$ is a circuit in some $2$-rigidity matroid then $G$ is $2$-connected. 
\end{prop}
\noindent
Using the same proof idea, we show below that circuits of connectivity $2$ are generated by smaller circuits via circuit elimination (we suspect this fact to be known, but could not find it stated explicitly in the literature).
\begin{lem}\label{lem:2circuitspartition}
Let $G$ be a circuit in some $2$-rigidity matroid, and suppose that there are vertices $u,w\in V(G)$ and subgraphs $X,Y\subseteq G$ such that $|X|\geq 2$, $|Y|\geq 2$, $X$ and $Y$ partition $G$, and $V(X)\cap V(Y)=\{u,w\}$. Then $e=\{u,w\}$ is not an edge of $G$. Morevoer, $X+e$ and $Y+e$ are circuits. 
\end{lem}
\begin{proof}
Suppose first that $e\in G$, and without restriction assume that $e\in Y$. 
Let $K$ be the complete graph on $V(Y)$. For the rank function $r$ of the given $2$-rigidity matroid we have $$r(K)=2v(Y)-3.$$
Since $X\cup Y$ is a circuit and $Y\subseteq K$, we have 
\begin{align}\label{eq:XandK}
r(X\cup K)&=r((X\cup Y)\cup K)\leq r(X\cup Y)+r(K)-r((X\cup Y)\cap K) \nonumber\\
&=r(X\cup Y)-r(Y)+r(K) =(|X|+|Y|-1)-|Y|+2v(Y)-3\nonumber \\
&=|X|+2v(Y)-4.
\end{align}
On the other hand, $X+e\subsetneq X\cup Y$, as $|Y|\geq 2$, which means $X+e$ is independent. So, we may successively apply $0$-extensions to all remaining vertices in $Y,$ resulting in an independent subgraph of $X\cup K$ of size 
$$|X|+1+2(v(Y)-2)=|X|+2v(Y)-3,$$
which is a contradiction. So, $e$ cannot be an edge of $G$.

Suppose now that $e\notin G$ and that $X+e$ is independent. 
Let $K$ be the complete graph on $V(Y)$ as before, and since $Y\subseteq K$, we again have~\eqref{eq:XandK}. On the other hand, $X+e$ is independent by assumption, and we may again successively apply $0$-extensions to all remaining vertices in $Y,$ resulting in an independent subgraph of $X\cup K$ of size 
$$|X|+1+2(v(Y)-2)=|X|+2v(Y)-3,$$
again a contradiction. Hence, $X+e$ is dependent, and so is $Y+e$ by symmetry. 

Therefore, there exist circuits $X'\subseteq X+e$ and $Y'\subseteq Y+e$, and note that $e$ is an edge in both, since $X$ and $Y$ are independent. So, we can write $X'=X''+e$ and $Y'=Y''+e$, where $X''\subseteq X$ and $Y''\subseteq Y$. 
Suppose for a contradiction that $X''\neq X$. Then applying circuit elimination with $X',Y'$ and $e$ gives that $X''\cup Y''$ is dependent, a contradiction since it is a proper subgraph of the circuit $G=X\cup Y$. So, we must have $X''=X$, meaning $X+e$ is a circuit. Symmetrically, $Y+e$ is also a circuit. 
\end{proof}
\begin{remark}
	From this it follows that the converse also holds. Suppose $X+e$ and $Y+e$ are circuits in a $2$-rigidity matroid, with $X\cap Y=\emptyset$ and $V(X)\cap V(Y)=e$. By circuit elimination, $X\cup Y$ is dependent, so there is a circuit $X'\cup Y'$ for some $X'\subseteq X$ and $Y' \subseteq Y$ with	
	$|X'|,|Y'|\geq 2$. Hence, by Lemma~\ref{lem:2circuitspartition}, $X'+e$ and $Y'+e$ are circuits. This implies $(X,Y)=(X',Y')$, and so $X\cup Y$ is a circuit. 
\end{remark}
\noindent
Let $\M$ be a $2$-rigidity family and suppose for a contradiction that $H$ is an $\M$-dependent connected cubic graph different from $K_4$ and $K_{3,3}$. Suppose further that, given $\M$, $H$ has the smallest number of vertices among all such graphs. Note that, by Lemma~\ref{lem:subcubic}, $H$ is an $\M$-circuit. Hence, as a consequence of Proposition~\ref{prop:circuits2conn} and Lemma~\ref{lem:2circuitspartition} we obtain
\begin{lem}\label{lem:3conn}
$H$ is $3$-connected. 
\end{lem}
\begin{proof}
By Proposition~\ref{prop:circuits2conn}, $H$ is $2$-connected. Suppose $H$ has a separating set of $2$ vertices $e=\{u,w\}$. Let $J$ be a component of $H-u-w$, let $X=H[V(J)\cup \{u,w\}]$, and let $Y=H\setminus X$.
Then $u,w$ and $X,Y$ are as in Lemma~\ref{lem:2circuitspartition}. Since $H$ is cubic, without restriction we may assume that $u$ has exactly one neighbour in $V(X)$, and note that, by $2$-connectivity, $w$ has at most two neighbours in $V(X)$. 

By Lemma~\ref{lem:2circuitspartition}, we have $e\notin H$ and $X+e$ is an $\M$-circuit. However, $X+e$ is a connected properly subcubic graph, which makes it $\M$-independent by Lemma~\ref{lem:subcubic} -- a contradiction.
\end{proof}
\begin{lemma}\label{lem:2circuit2conn}
	$H$ has a vertex that has no twins. 
\end{lemma}
\begin{proof}
	Take an arbitrary vertex $w$ and let $u_1,u_2,u_3$ be its neighbours. Note that the twin class of each $u_i$ must be a subset of $\{u_1,u_2,u_3\}$ as every twin of $u_i$ has to be a neighbour of $w$. The only way this can be achieved without a singleton twin class is when $\{u_1,u_2,u_3\}$ form a single twin class. In this case however, $H$ would contain $K_{3,3}$ as a subgraph, which is impossible for a connected cubic graph different from $K_{3,3}$.  
\end{proof}
\noindent
So, let $v\in V(H)$ be a vertex which does not have twins. Let $u_1,u_2,u_3$ be the neighbours of $v$. Since $H$ is cubic and $H\neq K_4, K_{3,3}$, there exists an edge $e=\{w_1,w_2\}\in H$, with $\{w_1,w_2\}\cap \{v,u_1,u_2,u_3\}=\emptyset$. Let $G=(H\oplus v)-e$. By Lemma~\ref{lem:tripling}, $G$ is $\M$-dependent.
\begin{lemma}\label{lem:cubicnoK4K33}
$G$ does not have a $K_4$ or $K_{3,3}$ subgraph. 
\end{lemma}
\begin{proof}
Let $v'$ be the newly added twin of $v$	in $G$. Since $G-v'=H-e$ does not have a $K_4$ or $K_{3,3}$ subgraph, the same holds for $G-v$, which is isomorphic. So, for $G$ to contain either subgraph, both $v$ and $v'$ must be used. This is not possible for a copy of $K_4$ since there is no edge between $v$ and $v'$. 

For the same reason if $G$ contained a copy of $K_{3,3}$, the vertices $v$ and $v'$ would be in the same partition class of this copy. Then the opposite class would be $\{u_1,u_2,u_3\}$, and let $v''$ be the remaining vertex of this $K_{3,3}$. 
Note that $v''$ has the same neighbours in $H$ as in $G$, and so it is a twin $v$ in $H$, contradicting our assumption.
\end{proof}
\noindent
Since, by Lemma~\ref{lem:3conn}, $H$ is $3$-connected, by Menger's theorem there exist three internally disjoint paths in $H$ between $w_1$ and $v$. Let us call them $P_1,P_2,P_3$. Without restriction we may assume that $u_i\in V(P_i)$ for $i=1,2,3$, and let 
$$U=V(P_1)\cup V(P_2)\cup V(P_3)\setminus \{v,u_1,u_2,u_3\}.$$ Note also that, since $\deg_H(w_1)=3$, the vertex $w_2$ must be the neighbour of $w_1$ on one of the $P_i$, and without loss of generality, let us suppose $w_2\in V(P_2)$. Observe that $v(G)=v(H)+1$, and that $G$ has two vertices of degree $2$, namely $w_1,w_2$, three of degree $4$, namely $u_1,u_2,u_3$, and the rest are of degree $3$. 

Hence, by the $0$-extension property applied iteratively along the paths $P_1,P_2-e$ and $P_3$, starting with $w_1$ and $w_2$, we obtain that $G'=G[V(G)\setminus U]$ is $\M$-dependent. So, $G'$ contains an $\M$-circuit $C$ as a subgraph. Note however that in $G'$ every vertex has degree at most $3$, so $C\subseteq G'$ must be cubic by Lemma~\ref{lem:subcubic}. Furthermore, since $C$ is a subgraph of $G$, by Lemma~\ref{lem:cubicnoK4K33} we have $C\neq K_4,K_{3,3}$. So, $C$ is connected, cubic, $\M$-dependent, not isomorphic to $K_4$ or $K_{3,3}$ and has strictly fewer vertices than $H$, since 
$$v(C)\leq v(G')\leq v(G-w_1-w_2)=v(G)-2=v(H)-1.
$$
This contradicts the minimality assumption on $v(H)$. Hence, $H$ as assumed cannot exist, and the proof of Theorem~\ref{thm:cubic} is completed.
\medskip

\noindent
By Corollary~\ref{cor:cubicorient} and Corollary~\ref{cor:redblue} we obtain the following structural statement about cubic graphs (note that closed trails in cubic graphs are cycles).
\begin{cor}
Any connected cubic graph $H\neq K_4,K_{3,3}$ can be red/blue coloured such that
\begin{itemize}
	\item The red and blue graphs $R$ and $B$ are forests,
	\item There exist tree orders $<_{_R}$ and $<_{_B}$ on $V(G)$ such that $v<_{_R} w$ iff $w<_{_B} v$, and 
	\item There is no colour-alternating directed cycle $$v_0>_{_R} v_1>_{_B} v_2>_{_R}\dots>_{_B} v_{2m}=v_0.$$ 
\end{itemize}
\end{cor}


\section{Concluding remarks}\label{sec:D}

\subsection*{Matroidal families vs matroids}
We remark that the na\"ive analogue of Theorem~\ref{thm:Laman1} for $2$-rigidity matroids for a fixed $n$ (as opposed to $2$-rigidity families) is false even under symmetry assumption: there exist symmetric $2$-rigidity matroids different from $\R_n$ in which every copy of $K_{3,3}$ is independent. An example for $n=6$ is the matroid whose bases are those of $\R_6$, excluding every copy of the triangular prism graph. That said, we are not aware of any further examples, and we checked by hand that (barring human error) there are none for $n=7$. In fact, Theorem~\ref{thm:Laman1} is `almost true' for symmetric $2$-rigidity matroids. Note that in the proof we only used heredity with one extra vertex, in Lemma~\ref{lem:tripling}. Therefore, it can be similarly shown that for any symmetric $2$-rigidity matroid $\M_{n}$ in which $K_{3,3}$ is independent, its restriction to (any) $n-1$ vertices is isomorphic to $\R_{n-1}$. In other words, graphs that are independent in $\R_n$ but not in $\M_n$ must be spanning on $[n]$. 
We are wondering if this distinction can occur for arbitrarily large $n$. 
\begin{conj}
There are only finitely many values of $n$ for which there exists a symmetric abstract $2$-rigidity matroid $\M_n\neq \R_n$ in which $K_{3,3}$ is independent. 
\end{conj}
\subsection*{Hierarchy of rigidity matroids}
For two (not necessarily graph) matroids $M$ and $N$ on the same ground set $E$, $N$ is said to be \emph{freer} than $M$, in notation $M\leq N$, if every $M$-independent set is $N$-independent. This equips any set of matroids on $E$ with a poset structure, and for the set of $2$-rigidity matroids on $E=\binom{[n]}{2}$, by Proposition~\ref{prop:lamanmax}, $\R_n$ is the unique maximum in this poset. It is natural to extend this notion to matroidal families: $\N$ is freer than $\M$ if every $\M$-independent graph is $\N$-independent. This again turns any set of matroidal families into a poset, and among the $2$-rigidity families $\R$ is the maximum. 

In particular, we have $\h<\R$. A central open problem in combinatorial rigidity is the Jackson-Tanigawa conjecture~\cite{Jac24} stating that $\h_n$ is the freest among $2$-rigidity matroids in which every copy of $K_{3,3}$ is a circuit. This would imply that $\h$ is the freest $2$-rigidity family in which $K_{3,3}$ is a circuit, and combined with our Theorem~\ref{thm:Laman1} this would make $\h$ the unique second-freest $2$-rigidity family. In order to make progress on this conjecture we would need a better understanding of $\h$, and we hope that our Theorem~\ref{thm:bernsteincharp} and its proof will be helpful in this matter.

We remark that there is a large body of work on related questions for $d\geq 3$, see~\cite{Cre23, CJJT25}. Unlike for $d=2$ it is not known in general that the freeness poset of all abstract $d$-rigidity matroids on $n$ vertices has a unique maximum. Whiteley~\cite{Whi96} conjectured that the maximum is attained by the so-called \emph{generic cofactor matroid}. This was recently confirmed for $d=3$ by Clinch, Jackson and Tanigawa~\cite{Clinch1} (see also \cite{Clinch2}). 
\subsection*{Weak saturation}
The author's interest in the topic stems from his work on weak saturation~\cite{btt21, CPSTZ, Shty, TT26}. 
For two graphs $G$ and $H$, $G$ is said to be \emph{weakly $H$-saturated} if the edges in $\binom{V(G)}{2}\setminus G$ can be ordered such that adding each next edge creates a new copy of $H$. The \emph{weak saturation number} of $H$, denoted $wsat(n,H)$, is the smallest number of edges in a weakly $H$-saturated graph on $n$ vertices.

Weak saturation was introduced by Bollob\'as~\cite{Bollobas} who determined $wsat(n,K_{d+2})$ for $d\leq 4$ to be $dn-\binom{d+1}{2}$, and conjectured this to hold for all $d$. This was confirmed indirectly by Lov\'asz~\cite{Lovasz77} in a paper that introduced the exterior algebra techniques in extremal combinatorics. First explicit proofs were given by Frankl~\cite{Frankl82} and Kalai~\cite{Kal84, Kal85}, and further proofs were found by
Alon~\cite{Alon} and Blokhuis~\cite{Blo90}. Typically, the weak saturation numbers are easy to guess and the upper bound constructions are natural, while the lower bound proofs are difficult. In almost all cases the proofs are not purely combinatorial, but use tools from algebra and geometry. 

Kalai~\cite{Kal85} formulated a principle allowing to search for lower bounds constructively:
if in some graph matroid $\M_n$ every copy of $H$ is a circuit then $wsat(n,H)\geq r(\M_n)$. In this light, observe that the matroids that can be used to determine $wsat(n,K_{d+2})$ are precisely the $d$-rigidity matroids, and under the natural additional assumptions  of symmetry and heredity one arrives at $d$-rigidity families. In particular, $2$-rigidity families are behind weak saturation with respect to $H=K_4$. 


\subsection*{Assumption of symmetry} Finally, let us remark that, while most $d$-rigidity matroids studied are symmetric (note that all $1$-rigidity matroids are), asymmetric example for $d\geq 2$ do exist. A common source would be rigidity matroids of non-generic bar-joint frameworks (e.g.~\cite{gra93textbook} Exercise 4.3). Another example from ~\cite{gra93textbook} (Corollary 4.6.1, Exercise 4.34), which is not an infinitesimal rigidity matroid is as follows. Take $\R_6$, pick a single copy of $K_{3,3}$ and declare it a circuit (instead of independent).
We are wondering if there are interesting examples of families of asymmetric $d$-rigidity matroids.



\bibliographystyle{abbrv}
\bibliography{refs}

\end{document}